\documentclass[11pt]{amsart}
\usepackage{amssymb, amscd, amsmath, array}

\theoremstyle{plain}
\newtheorem{thm}{Theorem}[section]

\newtheorem{conj}[thm]{Conjecture}
\newtheorem{cor}[thm]{Corollary}
\newtheorem{defn}[thm]{Definition}

\newtheorem{lem}[thm]{Lemma}

\newtheorem{que}[thm]{Question}
\newtheorem{rem}[thm]{Remark}

\def\C{{\mathbb C}}                               
\def\F{{\mathbb F}}                               

\def\L{\mathcal L}
\def\N{{\mathbb N}}                               
\def\O{{\mathcal O}}
\def\Q{{\mathbb Q}}                               
\def\Z{{\mathbb Z}}                               
\def\PP{{\mathbb P}}                               

\def\m{\mathfrak{m}}

\def\End{{\operatorname{End}}}                    
\def\GL{\operatorname{GL}}                        
\def\Gal{\operatorname{Gal}}                      
\def\Hom{\operatorname{Hom}}                      
\def\Ind{\operatorname{Ind}}
\def\ind{\operatorname{ind}}

\def\P{{\mathbb{P}}}

\def\Rep{\mathfrak{Rep}}

\def\SL{\operatorname{SL}}                        
\def\Sym{\operatorname{Sym}}
\def\Un{\operatorname{U}}                         
\def\Ps{\operatorname{Ps}}                         

\def\Fr{\operatorname{Fr}}                         
\def\det{\operatorname{det}}                      
\def\iso{{\stackrel{\sim}{~\longrightarrow~}}}    
\def\onto{{~\twoheadrightarrow~}}                 
\def\v2{{\vskip2mm}}

\numberwithin{equation}{section}

\newtheorem{Lem}[equation]{Lemma}

\begin{document}
\title{\sf Notes on modular representations of $p$-adic groups, \\ and the Langlands correspondence}
\author{\sf Dipendra Prasad}

\address{Tata Institute of Fundamental
Research, Colaba, Mumbai-400005, INDIA}
\email{dprasad@math.tifr.res.in}
\maketitle

\vskip10mm
\begin{quote}{\sf \footnotesize
These are expanded notes of some lectures given by the author for a workshop held at the Indian Statistical 
Institute, Bangalore in June, 2010, giving an exposition on 
the modular representations of finite groups of Lie type and $p$-adic groups,  and the modular Langlands correspondence. The aim of these lectures was 
to give an overview of the subject with several examples.
The author thanks Shripad M. Garge for writing and texing the first draft of these notes, and thanks 
U.K. Anandavardhanan as well as M.-F. Vign\'eras  for some comments on these notes.}
\end{quote}

\setcounter{tocdepth}{1}

\tableofcontents

{\sf In these notes, we will be considering representations of $G(k)$, where $G$ is a reductive algebraic group, e.g. $G = \GL_n$, 
and $k$ is a finite extension of  $ \Q_p, \,\F_p((t)),$ or $ \F_p$; in the first two cases, $k$ is called a $p$-adic field whose ring of integers will be denoted by $\O_k$, and we let $\varpi_k$ or sometimes just $\varpi$ denote a uniformizing parameter in $\O_k$.
The representation theory of these groups over $\C$ is a well trodden subject, and one can say that one understands a good deal about them. For finite groups of Lie type $G(\F_q)$, after the initial important work of J. Green for $\GL_n(\F_q)$, and 
Deligne-Lusztig who introduced geometric methods to representation theory, a rather complete understanding of the subject is  largely due to the work of G. Lusztig. For $G(k)$, $k$ a $p$-adic field, understanding representation theory of $G(k)$ is part of the Langlands program, pursued as such, or independently of it, by many people.

In these notes,  we will mostly consider analogous questions  for modular representations.
These representations are either mod $\ell$ ($\ell \ne p$) or mod $p$.
The answers are usually very different in these two cases.
Both are useful and interesting to pursue, and although the subject is of classical origins, it has attracted considerable 
renewed interest in light of number theoretic applications, such as to congruences of modular forms, and what has come to be called 
mod $\ell$ and mod $p$ Langlands program.
In general,  mod $\ell$ theory is  much better understood and tends to be much simpler than mod $p$. }

{\sf 
\section{Representations of $\GL_n(\F_q)$}
We begin with some generalities on the representations of $\GL_n(\F_q)$ in characteristic $p$ in the following well-known lemma.

\begin{lem}{\sf \em 
The following are equal:
\begin{enumerate}
\item The number of irreducible representations of $\GL_n(\F_q)$ in characteristic $p$.
\item The number of $p$-regular conjugacy classes in $\GL_n(\F_q)$.
\item The number of semisimple conjugacy classes in $\GL_n(\F_q)$.
\item The number of the characteristic polynomials of degree $n$ over $\F_q$ with non-zero constant term.
\item $q^{n-1}(q-1)$.
\end{enumerate}
}\end{lem}

\begin{rem}{\sf \em
Something very similar occurs for any reductive algebraic group over finite fields; for semisimple groups $G$,  the number of $p$-regular conjugacy classes in $G(\F_q)$ is $q^{rk(G)}$.
}\end{rem}

\v2
As a consequence, we deduce that all the irreducible mod $p$ representations of $\GL_2(\F_p)$ are $(\det)^j \otimes \Sym^i(V), 0 \leq j \leq p-2, 0 \leq i \leq p -1$, where $V$ is the standard $2$-dimensional representation of $\GL_2(\F_p)$.

\v2
If $q$ is a power of $p$, then there exists the automorphism, called the Frobenius automorphism, 
$$\Fr:\GL_2(\F_q) \iso \GL_2(\F_q) {\sf ~given~by~} X = (X_{ij}) \mapsto X^{(p)}:=(X_{ij}^p) .$$

\begin{lem}{\sf \em
Any irreducible representation of $\GL_2(\F_q)$, $q = p^d$, in characteristic $p$ is uniquely of the form
$${\displaystyle \chi \otimes \bigotimes_{j = 0}^{d - 1}\Fr^j(\Sym^{i_j}(V))}$$
for an integer $i$, $0 \leq i \leq p^d - 1$, where $i = i_0 + i_1 p + \cdots + i_{d - 1}p^{d - 1}$ is the $p$-adic expansion of $i$ with $0\leq i_j \leq p - 1$ and $\chi$ is a character of $\GL_2(\F_q)$ with values in $\F_q^\times$; in particular, any irreducible representation of $\GL_2(\F_q)$ in characteristic $p$ is defined over $\F_q$, and arises 
as the restriction of an irreducible algebraic representation of $\GL_2(\bar{\F}_q)$ (or more precisely,  of the group which is obtained from $\GL_2$ through 
the restriction of scalars
from $\F_q$ to $\F_p$).
}\end{lem}
\section{Reducing mod $\ell$}

One way of understanding modular representations is via reducing  representations in characteristic 0.
Via $\C \iso \bar{\Q}_\ell$, any vector space over $\C$ can be considered as one over $\bar{\Q}_\ell$, or 
better still,  over $E$, a finite extension of $\Q_\ell$. Representations of a group 
on a vector space over a field $E$ which is a finite extension of $\Q_\ell$ are called $\ell$-adic representations. 
Given an $\ell$-adic representation, there is the notion of a lattice in the
corresponding vector space over $E$, i.e.,  a finitely generated free $\O_E$-submodule $L$ of $V$ such that $L \otimes_{\O_E} E \iso V$.
For reduction mod $\m_E$, the maximal ideal of $\O_E$, one needs to choose an $\O_E$-lattice $L$ invariant under the finite group $G$ acting on $V$. For this,  choose
any lattice $L$ in $V$, and define, $\L := \sum gL$, which is a $G$-stable,  free $\O_E$-submodule of $V$.
Thus there are lattices $\L$ which are invariant under $G$.

\begin{defn}{\sf 
The reduction mod $\ell$ of an $\ell$-adic representation $V$ is the representation of the group $G$ on $\L/\m_E\L$ which 
 is a finite dimensional vector space over $k_E:=\O_E/\m_E$. This reduction mod $\ell$ depends on the choice of a lattice $\L$ invariant under $G$.
However, by a theorem due to  Brauer and Nesbitt, the semisimplification of the reduction mod $\ell$ is independent of choices.
}\end{defn}

\begin{rem}{\sf From the Brauer theory,  there is an obvious proof of the Brauer-Nesbitt theorem, since any two reductions have the same Brauer character.
}\end{rem}

\begin{que}{\sf Let 
$\pi$ be an irreducible $\bar{\Q}_\ell$-representation of a finite group $G$.
Suppose $\pi$ mod ${\ell}$ has two irreducible components $\pi_1$ and $\pi_2$.
Is there always a  lattice $\L_0$ such that $0 \to \pi_1 \to \L_0/\m_E\L_0 \to \pi_2 \to 0$ is a non-trivial extension of $\pi_2$ by $\pi_1$? 
Similarly,  is there a lattice $\L_1$ such that $0 \to \pi_2 \to \L_1/\m_E\L_1 \to \pi_1 \to 0$ is a non-trivial extension of $\pi_1$ by $\pi_2$? In particular, 
is it always true that ${\rm Ext}^1_{\bar{\F}_\ell[G]}[\pi_1, \pi_2] \not = 0$?
}\end{que}

As an example of the usefulness 
 of reduction mod $\ell$, we prove the following lemma.

\begin{lem}Let $G$ be a reductive algebraic group over a finite field $\F_q$ with $B = T\cdot U$ a Borel subgroup defined over $\F_q$, with $W = N(T)(\F_q)/T(\F_q)$ the relative Weyl group.
 Let $\Ps(\chi)$ be the principal series representation of $G(\F_q)$ induced from 
a character $\chi: T(\F_q) \rightarrow \bar{\F}_\ell^\times$, for $\ell \not = p$.
Then the following are equivalent.

\begin{enumerate}
\item The principal 
series representations $\Ps(\chi)$ and $\Ps(\chi')$ share a Jordan-H\"older factor 
with nonzero Jacquet module with respect to $U(\F_q)$.
\item The principal 
series representations $\Ps(\chi)$ and $\Ps(\chi')$ 
are the same in the Grothendieck group of representations of $G(\F_q)$.
\item The characters $\chi$ and $\chi'$ are conjugate under the relative Weyl group $W$.
\end{enumerate}
\end{lem}
\begin{proof} {\sf Observe that a character $\chi: T(\F_q) \rightarrow \bar{\F}_\ell^\times$
can be lifted to characteristic zero  $\widetilde{\chi}: T(\F_q) \rightarrow 
W_{\bar{\F}_\ell}^\times$,   where $W_{\bar{\F}_\ell}$  denotes the Witt ring of 
${\bar{\F}_\ell}$,   whose quotient field is contained in $\bar{\Q}_\ell$. 
The assertions in the Lemma are well-known in characteristic zero by calculating the
space of intertwining operators $\Hom_{G(\F_q)}[\Ps(\widetilde{\chi}), \Ps(\widetilde{\chi}')]$ using Frobenius reciprocity and the well-known calculation of Jacquet 
modules; crucial use is made of complete reducibility because of which
$\Ps(\widetilde{\chi})$ and $\Ps(\widetilde{\chi}')$ share a Jordan-H\"older factor if and only if there  is a nonzero intertwining operator between them.

From this characteristic zero theorem, it follows by reduction mod $\ell$ that if the characters $\chi$ and $\chi'$ are conjugate under the relative Weyl group $W$, then the principal 
series representations $\Ps(\chi)$ and $\Ps(\chi')$ are the same in the Grothendieck group of $\bar{\F}_\ell$-representations of $G(\F_q)$. Conversely, if the principal 
series representations $\Ps(\chi)$ and $\Ps(\chi')$ share a Jordan-H\"older factor with nontrivial Jacquet module with respect to $U(\F_q)$, then by a standard 
calculation of the Jacquet module (valid in characteristic $\ell$ too since Jacquet module is an exact functor for $\ell \not = p$), it follows that 
the characters $\chi$ and $\chi'$ are conjugate under the  Weyl group $W$.}
\end{proof}

\noindent{\bf Remark :} (a) A part of the  proof of the Lemma clearly works 
for $\ell=p$: it says that $\Ps(\chi)$ and $\Ps(\chi^w)$ are the same in the Grothendieck group of representations of $G(\F_q)$. Since  converse of this
uses Jacquet module techniques, it is not clear whether the 
Lemma holds  good for $\ell = p$.

(b) For $\ell \ne p$,  part 1 of the equivalent conditions in the Lemma about sharing a 
Jordan-H\"older factor  with nonzero Jacquet module with respect to $U(\F_q)$ is 
necessary for the proof given here. Since there are sub-quotients of a principal series 
in $+$ve characteristic which have trivial Jacquet module (such representations 
are usually called cuspidal representations), it is not clear to this author if we can do away with the hypothesis about `nonzero Jacquet module' for $\ell \ne p$.

\section{Reducing Deligne-Lusztig mod $\ell$}

Let $G$ be a reductive algebraic group over $\F_q$, and let $R(T_1,\theta_1)$ and $R(T_2,\theta_2)$ be two $\ell$-adic 
Deligne-Lusztig (virtual) representations 
associated to tori $T_1$ and $T_2$ inside $G$ which are defined over $\F_q$, and characters 
$\theta_1:T_1(\F_q) \rightarrow \bar{{\mathbb Z}}^\times_\ell$, $\theta_2:T_2(\F_q) \rightarrow \bar{{\mathbb Z}}^\times_\ell$. 
By reducing the characters
$\theta_1, \theta_2$ mod $\ell$ (which corresponds to going modulo the maximal ideal in $\bar{{\mathbb Z}}_\ell$), we get characters 
$\bar{\theta}_1:T_1(\F_q) \rightarrow \bar{{\mathbb F}}^\times_\ell$, $\bar{\theta}_2:T_2(\F_q) \rightarrow \bar{{\mathbb F}}^\times_\ell$. 
By the character formulae for Deligne-Lusztig representations, cf. Theorem 4.2 of [DL], and Brauer theory of characters (according to which two representations in characteristic $\ell$ have the same Jordan-H\"older 
composition factors if and only if their Brauer characters are the same), it is clear that the semi-simplification of 
the reduction mod $\ell$ of a Deligne-Lusztig representation $R(T,\theta)$, depends only on the reduction of $\theta$ mod $\ell$.  In particular, it makes good sense 
to talk of the Deligne-Lusztig representation $R(T,\theta)$ in the Grothendieck group of representations over $\bar{\F}_\ell$ of $G(\F_q)$ for $\theta$ a character ${\theta}:T(\F_q) \rightarrow \bar{{\mathbb F}}^\times_\ell$ where $\ell= p$ is allowed. 

This trivial 
observation is already useful to prove that the reduction mod $\ell$ of certain $R(T,\theta)$ is not irreducible: 
write the character $\theta = \theta_\ell \cdot \theta'_\ell$
as product of two characters $\theta_\ell$ and $\theta'_\ell$, the first having order a power of $\ell$, the
second of order coprime to $\ell$; such a decomposition $\theta = \theta_\ell \cdot \theta'_\ell$ could be called  the mod $\ell$ 
Jordan decomposition of the character $\theta$.
 Then, $\bar{\theta} = \bar{\theta}'_\ell$, and therefore by the earlier observation, $R(T,\theta)$ and $R(T,\theta'_\ell)$ have 
the same reduction mod $\ell$. Therefore if $R(T,\theta'_\ell)$ is a reducible representation (in characteristic 0), and not just a reducible virtual 
representation,  so will $R(T,\theta)$ be 
mod $\ell$. This procedure allows one to prove that the reduction mod $\ell$ of certain $R(T,\theta)$ is not irreducible; its success 
partly depends on knowing when $R(T,\theta)$ is an honest representation of $G(\F_q)$, and not just a virtual representation. (This happens for 
example for $T$ a split torus and any $\theta$, or for $\theta$ in general position for any $T$, or a combination of these two 
via induction in stages.) This  raises
another interesting question: can an $R(T,\theta)$ be identically zero mod $\ell$ for some $\ell$?

We recall that among mod $\ell$ representations of $\GL_2(\F_p)$ ($\ell \not = p$), there is one representation which seems at first sight to be 
somewhat of an anomaly. To define this, note that the Steinberg representation $St$ of $\GL_2(\F_p)$ is realized on the space of functions on $\PP^1(\F_p)$ 
with values in $\F_\ell$ modulo the constant functions. There is a natural map on the space of functions on $\PP^1(\F_p)$ 
to $\F_\ell$ obtained by sending 
a function on $\PP^1(\F_p)$ to the sum of its values on $\PP^1(\F_p)$. Under this map, the constant function 1 
goes to $(p+1)$, hence if $\ell|(p+1)$, it gives a natural map from $St$ to $\F_\ell$ whose kernel is a representation of $\GL_2(\F_p)$ of dimension $(p-1)$, and which is a  cuspidal representation
in that its Jacquet module is zero, but it still appears as a sub-quotient of a principal series representation. In fact, this representation
of $\GL_2(\F_p)$ of dimension $(p-1)$ is the reduction mod $\ell$ of any cuspidal representation  of $\GL_2(F_p)$   in characteristic zero which has 
 trivial central character and which arises 
from a character $\chi:\F_{p^2}^\times \rightarrow \bar{\Z}^\times_\ell$ 
whose order is a power of $\ell$. (This
 follows by considering 
the Brauer character of the two representations involved.)  This example,
and the observations in the previous paragraph,  motivates  us to ask the following general question.

\begin{que} Let $G$ be a reductive algebraic group over $\F_q$, and let $R(T_1,\theta_1)$ and $R(T_2,\theta_2)$ 
be two Deligne-Lusztig representations mod $\ell$
associated to tori $T_1$ and $T_2$ inside $G$ defined over $\F_q$, and characters 
$\theta_1:T_1(\F_q) \rightarrow \bar{{\mathbb F}}^\times_\ell$, $\theta_2:T_2(\F_q) \rightarrow \bar{{\mathbb F}}^\times_\ell$. 
Then is it true that
$R(T_1,\theta_1)$ and $R(T_2,\theta_2)$ share a common Jordan-H\"older factor if and only if 
$\theta_1:T_1(\F_q) \rightarrow \bar{{\mathbb F}}^\times_\ell$, $\theta_2:T_2(\F_q) \rightarrow \bar{{\mathbb F}}^\times_\ell$ are geometrically conjugate, i.e., 
conjugate when considered as characters on $T_1(\F_{q^n})$ and $T_2(\F_{q^n})$ for some $n$? Since two characters 
$\theta_1,\theta_2:  T(\F_q) \rightarrow \bar{\F}_\ell^\times$ are conjugate by $G(\F_q)$ if and only if they are geometrically conjugate,  a special case
of the question here for the case $T_1 = T_2$ would assert that the mod $\ell$ representations 
$R(T, \theta_1)$ and $R(T,\theta_2)$ share a Jordan-H\"older factor if and only if $\theta_1$ and $\theta_2$ are conjugate.
\end{que}

We end the section with a simple application of the ideas here.

\begin{lem}
\begin{enumerate} 
\item A cuspidal $\bar{\Q}_\ell$-representation of $\GL_n(\F_q)$ 
remains irreducible mod $\ell$ for $\ell \ne p$.
\item The Steinberg representation of $\GL_n(\F_q)$ contains a cuspidal representation mod $\ell$ in its
Jordan-H\"older factor if  $\ell| \frac{q^n-1}{q-1}$.
\end{enumerate}
\end{lem}

\begin{proof} {\sf For part $(a)$, it suffices to observe that a cuspidal $\bar{\Q}_\ell$-representation of $\GL_n(\F_q)$ when restricted 
to the mirabolic subgroup $P_0(\F_q)$ of $\GL_n(\F_q)$ consisting of those elements in $\GL_n(\F_q)$ with the last row equal to 
$(0,\cdots, 0, 1)$ can be written 
as  induction of a non-degenerate character $\psi: U(\F_q) \to \bar{\Z}_\ell$ to $P_0$. We can assume that the reduction of $\psi$ still gives a 
non-degenerate character of $U(\F_q) \to \bar{\F}_\ell$, which then proves that the reduction mod $\ell$ of a cuspidal
$\bar{\Q}_\ell$ representation remains irreducible when restricted to $P_0(\F_q)$.

We now prove part $(b)$. 

Let $T$ be the torus inside $\GL_n(\F_q)$ corresponding to $\F_{q^n} ^\times \hookrightarrow  \GL_n(\F_q)$. Let $\F^1_{q^n}$ be the subgroup
of $\F_{q^n}^\times$ consisting of elements of $\F_{q^n}^\times$ whose norm to $\F_q^\times$ is 1. The order of $\F^1_{q^n}$ is $\frac{q^n-1}{q-1}$. Therefore
if $\ell| \frac{q^n-1}{q-1}$, there exists a character $\theta^1$ of $\F^1_{q^n}$ of order $\ell$. Extend the character $\theta^1$ of $\F^1_{q^n}$
to a character $\theta$ of $\F^\times_{q^n}$ of order  a power of $\ell$. Clearly $\theta$ is then a regular character, i.e., its conjugates 
under the Galois group (=Weyl group) are distinct. Thus $R(T,\theta)$ is a cuspidal representation of $\GL_n(\F_q)$ which by part $(a)$ remains irreducible mod $\ell$. Since $\theta$ is of order a power of $\ell$, its reduction mod $\ell$ must be trivial. Therefore, 
$R(T,\theta)$ and $R(T,1)$ 
are the same representations mod $\ell$ in the Grothendieck group of representations of $\GL_n(\F_q)$ over $\bar{\F}_\ell$.  

It is now well-known that the Deligne-Lusztig
virtual representation $R(T,1)$ contains the Steinberg representation over $\C$, and the other components of $R(T,1)$ being non-generic. 
Looking at the equality of representations $R(T,\theta)$ and $R(T,1)$ 
in the Grothendieck group of representations of $\GL_n(\F_q)$ over $\bar{\F}_\ell$, we find that the irreducible representation
$R(T,\theta)$ mod $\ell$ must be a Jordan-H\"older factor of the Steinberg representation mod $\ell$.}
\end{proof}
  
\section{A comparative study of representations of $\GL_2(\F_p)$ in different characteristics}
We begin by summarizing some properties about the irreducible representations of $\GL_2(\F_p)$ over $\C$ in the following table.

\vskip5mm
\begin{center}
\begin{tabular}{|l|c|l|}
\hline
& dimension & parametrization \\ \hline
1-dimensional & 1 & $\chi:\F_p^\times \to \C^\times$ \\ \hline
Steinberg & $p$ & $\chi:\F_p^\times \to \C^\times$ \\ \hline
Principal series & $p+1$ & $(\chi_1, \chi_2):\F_p^\times \to \C^\times$, $\chi_1 \ne \chi_2$ \\ \hline
cuspidal & $p - 1$ & $\chi:\F_{p^2}^\times \to \C^\times$, $\chi \ne \chi^p$. \\ \hline
\end{tabular}
\end{center}
\vskip5mm

Here the 1-dimensional representations are simply characters $\chi \circ \det:\GL_2(\F_p) \stackrel{\det}{\to} \F_p^\times \stackrel{\chi}{\to} \C^\times$.

To define the  Steinberg representation,  observe that $\GL_2(\F_p)$ operates on $\P^1(\F_p)$; the Steinberg representation St is the quotient of the corresponding permutation representation by the trivial 
representation.

The Principal series representations are defined to be the irreducible representations 
 $\Ind_B^{\GL_2(\F_p)}(\chi_1, \chi_2)$ where $(\chi_1, \chi_2): B(\F_p) \to \C^\times$, is given by $\begin{pmatrix} a & b \\ 0 & d \end{pmatrix} \mapsto \chi_1(a) \chi_2(d)$,
and we demand $\chi_1 \not = \chi_2$ to ensure irreducibility.

\v2
\noindent{$\bullet$ Reduction mod $\ell \ne p$.}
(In the following tables $\boxplus$ means addition in the Grothendieck group.)

\vskip5mm
\begin{tabular}{|l|l|l|}\hline
& & \\
1-dimensional & $\chi:\F_p^\times \to \bar{\Z}_\ell^\times$ & $\bar{\chi}$ \\
& & \\
 \hline
& & \\
Steinberg & $\chi:\F_p^\times \to \bar{\Z}_\ell^\times$ & $\begin{matrix} \bar{\chi} (1 \boxplus (p - 1)) \iff \ell| (p+1) \\ \bar{\chi}(p) \iff \ell \not | (p+1) \end{matrix}$ \\  & & \\ \hline
& & \\
Principal series & $(\chi_1, \chi_2):\F_p^\times \to \bar{\Z}_\ell^\times$, $\chi_1 \ne \chi_2$ & remains irreducible mod $\ell$ $\iff \bar{\chi_1} \ne \bar{\chi_2}$; \\
& & if $\bar{\chi_1} = \bar{\chi_2}$,  then revert back to Steinberg. \\ & & \\ \hline
 & & \\
cuspidal &$\chi:\F_{p^2}^\times \to \bar{\Z}_\ell^\times$, $\chi \ne \chi^p$ & 
Always irreducible mod $\ell$! This follows since any irreducible \\ & & $\bar{\F}_\ell$-representation of $\GL_2(\F_p)$ must have a Whittaker\\
& & model, and therefore must have dimension $\geq p-1$.  \\
& &  \\\hline
\end{tabular}

\vskip5mm
\noindent $\bullet$ Reduction mod $p$.

\vskip5mm
\begin{tabular}{|l|l|l|}\hline
& & \\
1-dimensional & $\chi:\F_p^\times \to \Z_p^\times$ & $\bar{\chi}$  remains irreducible.  \\ & & \\ \hline & & \\
Steinberg & $\chi:\F_p^\times \to \Z_p^\times$ & also remains irreducible. \\ & & \\ \hline & & \\
Principal series & $(\chi_1, \chi_2):\F_p^\times \to \Z_p^\times $, $\chi_1 \ne \chi_2$ & never irreducible mod $p$; \\
& & in fact sum of $2$ irreducibles. \\ & & \\ \hline & & \\
cuspidal & $\chi:\F_{p^2}^\times \to \bar{\Z}_p^\times$, $\chi \ne \chi^p$ & irreducible mod $p$ if $\chi(x)= x^a\bar{x}^b$, with $|a-b|=1$,   \\ 
& & 
and sum of $2$ irreducibles otherwise. \\ & & \\ \hline 
\end{tabular}

\vskip8mm
\noindent The following two lemmas describe the reduction mod $p$ of a principal series as well as of cuspidal representations of $\GL_2(\F_p)$. 
Recall that the principal series representations are  parametrized by characters $\chi_1, \chi_2: \F_p^\times \to \Z_p^\times \onto \F_p^\times$, which are represented  by  a pair of integers $(i_1, i_2)$ 
where an integer $i$ denotes the map $x\rightarrow x^i$, on $\F_p^\times$.

\begin{lem} There are exact sequence,

$$ 0 \rightarrow \det^{i_1} \otimes \Sym^{i_2-i_1} V \rightarrow  \Ps(i_1, i_2) \rightarrow 
 \det^{i_2} \otimes \Sym^{(p-1) - (i_2-i_1)} V \rightarrow 0, ~~~{\rm if  }~~~~i_2 \geq  i_1  $$

$$ 0 \rightarrow \det^{i_1} \otimes \Sym^{(p-1)-(i_1-i_2)} V \rightarrow  \Ps(i_1, i_2) \rightarrow 
 \det^{i_2}\otimes  \Sym^{i_1-i_2} V \rightarrow 0, ~~~{\rm if  }~~~~i_1 \geq  i_2  .$$
Both the exact sequences are non split if $i_1 \not = i_2$.
\end{lem}

\v2

\begin{lem} For a cuspidal representation $Ds(\chi)$   corresponding to the character $\chi: x \rightarrow x^a \bar{x}^b$  from $\F_{p^2}^\times \rightarrow \F_{p^2}^\times$ 
(where $x$ is the identity map from $\F^\times_{p^2}$ to $\F^\times_{p^2}$, 
or rather its Teichm\"uller lift to the Witt ring from $\F^\times_{p^2} \to W_{\F_{p^2}}^\times$, 
and  $\bar{x} = x^p$), its reduction mod $p$, assuming without loss of generality 
$a >b$,  has Jordan-H\"older factors,
$\det^{b+1} \otimes \Sym^{a-b-2}(V)$  and $ \det^a \otimes \Sym^{p-1-(a-b)}(V)$. (Without an
explicit model over $\bar{\Z}_p$ for $DS(\chi)$, one cannot say more about 
reduction mod $p$.)
\end{lem}

Both the lemmas  can be checked via Brauer characters.

\v2

\noindent{\bf Remark :} It may be noted that for $a-b = 1$, one of the Jordan-H\"older factors of $Ds(\chi)$ is $\Sym^{-1}(V)$ which under the usual 
convention is the zero vector space.

\v2

\noindent{\bf Remark :} One general method to understand cuspidal representations $Ds(\chi)$ is via the method of {\it basechange} which allows
one to think about these representations in terms of principal series representations $\Ps(\chi, \bar{\chi})$ of $\GL_2(\F_{p^2})$ where
$\bar{\chi}(x) = \chi(\bar{x})= \chi^p$. The method of basechange is valid in the modular case too, and allows us to relate 
lemmas  4.1 and 4.2, which we briefly review. Basechange identifies irreducible 
representations of $\GL_2(\F_{p^2})$ 
which are invariant under the Galois group which then extend to a representation of $\GL_2(\F_{p^2}) \rtimes \Z/2$ 
with representations of $\GL_2(\F_p)$. The method of basechange is based on the 
Shintani character identity, according to which the character  of the extended representation  of $\GL_2(\F_{p^2}) \rtimes {\Fr}^{\Z/2}$ 
at the non-identity component $(g, \Fr)$ for $g \in \GL_2(\F_{p^2})$ is the same as that of a representation of $\GL_2(\F_p)$ at the element
$Nm(g)$ of  $\GL_2(\F_p)$ which belongs to the $\GL_2(\F_{p^2})$-conjugacy class defined by $g \cdot \Fr(g)$. Given that the Shintani 
basechange identity
holds for $\C$-representations, and therefore for $\bar{\Q}_\ell$-representations, choosing a $\GL_2(\F_{p^2}) \rtimes {\Fr}^{\Z/2}$ invariant lattice, it holds good  for mod $\ell$ representations too, including $\ell =p$. We also recall (an easy check) that the base change of an algebraic representation  $V$
of $G(\F_p)$ (i.e., an algebraic representation of $G(\bar{\F}_p)$ restricted to 
$G(\F_p)$) to $G(\F_{p^2})$, for $G$ any connected algebraic group over $\F_p$, is the representation of 
$G(\F_{p^2})$ on the space $V \otimes V^{\Fr}$ which as a vector space is the same as $V \otimes V$ on which $g \in G(\F_{p^2})$ acts as
$g(v_1 \otimes v_2) = gv_1 \otimes \Fr(g)v_2$. 

To relate Lemma 4.1 with 4.2, we will actually need a form of Lemma 4.1 for $\GL_2(\F_{p^2})$ which will actually have four  Jordan-H\"older factors,
two of which arise as basechange of the two Jordan-H\"older factors in Lemma 4.1, and two of which are Galois conjugate representations, 
and therefore do not contribute to the twisted trace. We leave details including the description of the Jordan-H\"older factors 
of a principal series representation of $\GL_2(\F_{p^2})$ in characteristic $p$ in terms of Lemma 1.3 to the reader.

\section{Generalities  on representations of $G(k)$}

We make some remarks on representations of $G(k)$ where $G$ is a reductive group defined over $k$,  a $p$-adic field. The groups $G(k)$ are  locally compact topological groups 
which are totally disconnected. They come equipped with filtration by 
compact open subgroups, such as for $G(k)= \GL_n(k)$, by the principal congruence subgroups,
$$K_m = \{g \in \GL_n(\O_k):g \equiv 1 \pmod{\varpi^m}\} .$$
\v2
A {\em smooth representation} of $G=G(k)$ is a representation of $G$ 
as an abstract group on a vector space $V$ such that for all $v \in V$ 
there exists a compact open subgroup $K_v \subseteq G$ such that 
$K_v$ operates trivially on $v$, i.e., if we consider the discrete topology on 
$V$, then this action is continuous.

\v2
A smooth representation $\pi$ of $G$ is said to be {\em admissible}  if $\pi^K = \{v \in \pi: kv = v ~\forall k \in K\}$ is finite dimensional for any compact open subgroup $K$ of $G$.
Thus for $\pi$ an admissible representation, we have
$$\pi = \bigcup_K \pi^K,$$
where the $K$'s are compact open subgroups of $G(k)$, each $\pi^K$ being finite dimensional.

The main question in the  representation theory  of $p$-adic groups is:
\begin{que}
Classify irreducible admissible representations of $G(k)$.
\end{que}

We recall the notion of parabolic induction.
We have a parabolic subgroup $P =MN \supseteq B$.
Given a representation $(\rho, V)$ of $M$, consider $\rho$ as a representation of $P$ by extending $\rho$ trivially across $N$.
Then 
$$\Ps(\rho) := \Ind_P^G \rho: = \left\{f : G \to \C: f {\sf ~locally~constant}, f(pg) = \rho(p)f(g)\right\} .$$
The group $G$ operates on this space by right translations.

\begin{thm}
$\Ind: \Rep(M) \to \Rep(G)$ takes admissible representations to admissible representations and finite length representations to finite length representations, and holds in all the cases being considered in these notes, i.e., 
for representations over $\C, \bar{\F}_\ell, \bar{\F}_p$.
\end{thm}

The part of the theorem about admissible representations going to admissible 
representations under parabolic induction is a simple result which works uniformly for representations over $\C, \bar{\F}_\ell, \bar{\F}_p$ and follows from the generality that 
for any compact open subgroup $K$ of $G$, and a maximal parabolic $P$,
$K\backslash G/P$ is a finite set. However, the part about finite length 
representations going to finite length representations is non-trivial, and done,  it seems,
separately in all the three cases of  $\C, \bar{\F}_\ell, \bar{\F}_p$.

\vspace{4mm}

These representations $\Ind_P^G \rho$, $P \ne G$ are called {\em principal series representations}, and give a large collection of representations of $G(k)$.
The representations of $G(k)$ that do not occur as sub-quotients of principal series 
representations are called {\em supercuspidal representations} of $G(k)$.

\v2
\noindent({\sc Harish-Chandra philosophy}) Representation theory of $G$ is split in 2 parts:
\begin{enumerate}
\item Understand sub-quotients of principal series representations;
\item understand those representations of $G$ which do not arise as sub-quotients of 
principal series representations  for any $P \ne G$.
\end{enumerate}

\begin{conj}[\sf Local Langlands correspondence (LLC) $\sim$ 1966]{\sf \em
There exists a bijective correspondence between irreducible $\C$-representations of $\GL_n(k)$ and $n$-dimensional $\C$-representations of the Weil-Deligne group $W'_k = W_k \times \SL_2(\C)$, 
where $W_k$ is the Weil group of $k$ and is a variant of the Galois group of $k$.
}\end{conj}

The Local Langlands correspondence 
is valid for $n=1$ by the
 {\em local class field theory} according to which  $\Gal(\bar{k}/k)^{ab} \cong \widehat{k}^\times ={\O}_k^\times \times \widehat{\Z} .$

\v2
The Local Langlands correspondence  was proved around 2000 by M.Harris and R.Taylor [HT], 
and independently by G.Henniart [Hen].

\v2

The representations that are not sub-quotients of any principal series are called 
{\it supercuspidal representations}. There is a weaker notion of a {\it cuspidal representation}  where one demands that the representation does not
appear as a sub-representation of a principal series representation; this is equivalent by Frobenius reciprocity to demanding that all the Jacquet modules are zero. Over $\C$,
the notion of cuspidal and supercuspidal representations are the same, but not with $\bar{\F}_\ell$ or $\bar{\F}_p$ coefficients, even for 
finite groups $G(\F_q)$.

Supercuspidal representations serve as building blocks of all representations by theorem 5.2. 
which is part of a very standard
theory for representations of $G(k)$ over $\C$, and was extended by Vigneras to  $\ell \ne p$.
In a recent work, Florian Herzig [Her] 
has  proved  an analogous theorem for $\ell = p$ too, and along the way, clarified the role of what was called a 
supersingular representation defined in terms of certain Hecke algebras by Barthel and Livne. By the work of Herzig, supersingular becomes identical to 
 supercuspidal representations but which are yet not classified. One of the important results that Herzig proves along the way are that there
are very few principal series which are reducible except the Steinberg and generalized Steinberg representations. The mod $p$ theory lacks symmetry:  $\pi_1 \times \pi_2$, even if irreducible,  need not be $\pi_2 \times \pi_1$. The simplest example to see this is for $\GL_2(\Q_p)$ and the principal series
representation $\chi_1 \times \chi_2$ for characters $\chi_i:\Q_p^\times \to \F_p^\times$
whose restriction to $\Z_p^\times$ are distinct characters which factor through 
$\bar{\chi}_i: \Z_p^\times/(1+p \Z_p) = \F_p^\times \to \F_p^\times$. From the Iwasawa decomposition, $\GL_2(\Q_p) = 
\GL_2(\Z_p) \cdot B$,  the restriction of the principal series 
$\Ps(\chi_1, \chi_2) := \chi_1 \times \chi_2$ to $\GL_2(\Z_p)$ is the induced representation from $B(\Z_p)$
to $\GL_2(\Z_p)$ from the character $(\bar{\chi}_1, \bar{\chi}_2)$. Therefore the space of
$K(1)$ fixed vectors inside the principal series $\chi_1 \times \chi_2$, where $K(1)$ 
is the principal congruence subgroup in $\GL_2(\Z_p)$ of level 1, is the same as the 
principal series representation of $\GL_2(\F_p)$ induced from the character 
$(\bar{\chi}_1, \bar{\chi}_2)$ of $B(\F_p)$. But by Lemma 3.1, we know that the two principal series
representations of principal series representation of $\GL_2(\F_p)$ induced from the characters  $(\bar{\chi}_1, \bar{\chi}_2)$, and $(\bar{\chi}_2, \bar{\chi}_1)$ are distinct. 
Therefore the principal series representations $\chi_1 \times \chi_2$ and $\chi_2 \times \chi_1$ of $\GL_2(\Q_p)$ are not isomorphic.

One of the first tools Herzig develops is the structure of certain Hecke algebras which we briefly recall.
Let $G$ be a split connected reductive group over a $p$-adic field $k$. 
For $K$ a hyperspecial maximal compact open subgroup of $G$, a $K$-weight is an irreducible finite-dimensional $\bar{\F}_p$-representation $V$ of $K$ 
(which factors through the reductive quotient of $K$).

Herzig considers the Hecke algebra $\mathcal{H}_G(V)
=\End_G(\ind_K^GV)$  
(where $\ind_K^GV$ denotes compact induction)  for  any irreducible $K$-weight $V$ 
and proves a  mod $p$ Satake isomorphism, identifying $\mathcal{H}_G(V)$ 
to an explicit commutative algebra. 
Then, for any irreducible $\bar{\F}_p$-representation $\pi$ of $G$, 
there is a $K$-weight $V$ and an algebra homomorphism $\chi:\mathcal{H}_G(V) \to \bar{\F}_p$ such that $\pi$ occurs as a quotient of
$\ind^G_KV \otimes_{\mathcal{H}_G(V)} {\bar{\F}}_p$, where $\mathcal{H}_G(V)$ operates on $\bar{\F}_p$  via $\chi$.

\v2

We close this section with a natural question which we do not know if it is answered.

\v2

\noindent{\bf Question :} Let $G$ be a reductive group over a $p$-adic field $k$,  $P=MN$ a proper 
parabolic subgroup of $G$, $\pi$ an irreducible representation of $G(k)$ over $\bar{\F}_p$,
and $\pi'$ an irreducible representation of $M(k)$ over $\bar{\F}_p$. Then by the Frobenius reciprocity,
$$\Hom_{G(k)}[\pi, \Ind_{P(k)}^{G(k)} \pi'] = \Hom_{M(k)}[\pi_N,  \pi'].$$
It follows that the Jacquet module of a principal series representation is always non-zero. Is it true that the
Jacquet module of a supercuspidal representation is always zero, or could it happen that the Jacquet module is nonzero,
but it has no irreducible quotient? This question is there only because one does not know if the Jacquet module takes a 
representation of finite length to a representation of finite length.

\section{A basic argument in characteristic $p$}
{\sf In this section, we are  in characteristic $p > 0$.

\begin{lem}
Any finite dimensional representation $V$ of a finite $p$-group $G$ over a field $F$ of characteristic $p$ has a nonzero 
fixed vector.
\end{lem}

\begin{proof} {\sf Let $V$ be a representation space for $G$ of dimension $d$. Let $v$ be a nonzero vector in $V$, and consider the $\F_p$-span, say $W$,  of $\{gv\}$ as $g$ varies over $G$. Then $W$ is a finite dimensional vector space over $\F_p$ which is $G$-invariant. 
It suffices then to assume that $F=\F_p$, and $V=W$, in which case there are
$p^d-1$ nonzero elements in  $V$.
Now deduce a contradiction by looking at the action of $G$ on
nonzero elements  in $V$, and noting that} $ p^d - 1 \ne 0 \pmod{p}$.
\end{proof}

\begin{cor}
Any irreducible representation of $G=\GL_n(\F_q)$ 
in characteristic $p$ is a quotient of a principal series induced from a character on a Borel subgroup of $G=\GL_n(\F_q)$. 
\end{cor}

\begin{proof}
{\sf $B = T \cdot U$ where $U$ is a $p$-group.
Look at $\pi^U \ne 0$, the Jacquet functor, as a $T$ module, thus contains an irreducible character $\chi$ of $T$.
Clearly $\pi$ is a quotient of $\Ind_B^G \chi$.}
\end{proof}
}

\begin{cor}Any representation $\pi$ over $\bar{\F}_p$ of $\GL_n(k)$,  $k$ a $p$-adic field, has a vector fixed
by the principal congruence subgroup $K(1)$ of $\GL_n(\O_k)$, and therefore contains 
an irreducible  representation of $\GL_n(\O_k)$ which factors through 
$\GL_n(\O_k/\varpi_k)$ as a submodule. 
(Such representations of $\GL_n(\O_k/\varpi_k)$ are called Serre weights of $\pi$.)
\end{cor}

\begin{cor} If $V$ is a finite dimensional irreducible representation of 
$K= \GL_n(\O_k)$,
then the compact induction $\pi = {\rm ind}_{K\cdot k^\times}^{\GL_n(k)} V$ 
is of infinite length.
\end{cor}

\begin{proof} {\sf We prove that the space of intertwining operators, $\End_G(\pi) = 
\Hom_G[ {\rm ind}_{K\cdot k^\times}^{\GL_n(k)} V, {\rm ind}_{K\cdot k^\times}^{\GL_n(k)} V]$ 
is infinite dimensional. By Frobenius reciprocity,
$$\Hom_G[ {\rm ind}_{K\cdot k^\times}^{\GL_n(k)} V, 
{\rm ind}_{K\cdot k^\times}^{\GL_n(k)} V] = \Hom_{K\cdot k^\times}[V, {\rm ind}_{K\cdot k^\times}^{\GL_n(k)} V].$$ 

We now recall the Cartan decomposition according to which $\GL_n(k) = KAK$, where $A$ 
is the semigroup of diagonal matrices of the form $\varpi_{\underline{m}}= (\varpi^{m_1},\varpi^{m_2},\cdots, \varpi^{m_n})$ with $m_1 \geq m_2 \geq \cdots \geq m_n$.

From the Cartan decomposition, it follows that the restriction of 
$\pi = {\rm ind}_{K\cdot k^\times}^{\GL_n(k)} V$ to $K$ is a direct sum of induced representations, ${\rm ind}_{K_{\underline{m}}}^K V^{\underline{m}}$ where $K_{\underline{m}} = K \cap \varpi_{\underline{m}} K \varpi^{-1}_{\underline{m}} $, and $V^{\underline{m}}$ is the representation of 
 $K_{\underline{m}}$ through the map,  $K_{\underline{m}} \to K $ given by 
$x \to \varpi_{\underline{m}}^{-1}\cdot x \cdot \varpi_{\underline{m}} $. Since the representation $V$ of $K$ factors through $\GL_n(\F_q)$, this allows one to get infinitely many intertwining operators, as is easily checked. } 
\end{proof}
 
\section{Highest weight modules}
{\sf For finite groups of Lie type $G(\F_p)$, such as $\GL_n(\F_p)$, in the case $\ell = p$,  methods of algebraic groups/algebraic geometry 
can be brought to bear on the problem of constructing or understanding 
representations of $G(\F_p)$.

Let $G = \GL_n(\bar{\F}_p)$ and let $\pi$ be an algebraic representation of $G$.
Let $U$ be the group of upper triangular unipotent matrices and let $B=T\cdot U$ be the group of upper triangular matrices.

\begin{defn}{\sf
A  representation $\pi$ is  said to be a {\em highest weight module} with weight $\underline{\lambda}:= \lambda_1 \geq \cdots \geq \lambda_n$ 
if there exists $v \in \pi^U$ such that the maximal torus $T$ in $B$ operates by $\underline{t}v = \underline{t}^{\underline{\lambda}} v$
where $\underline{t}^{\underline{\lambda}} = t_1^{\lambda_1} \cdots t_n^{\lambda_n}$.
We further demand that $\pi$ is generated as a $G$-module by this vector.
}\end{defn}

\begin{thm}{\sf \em 
Any irreducible representation $\pi$  of $\GL_n(\bar{\F}_p)$  
is a highest weight module for a unique weight $\underline{\lambda}(\pi):= \lambda_1 \geq \cdots \geq \lambda_n$, and the association $\pi \rightarrow \underline{\lambda}(\pi)$ gives a bijective correspondence between irreducible representations of $\GL_n(\bar{\F}_p)$, and highest weights   
$\underline{\lambda}:= \lambda_1 \geq \cdots \geq \lambda_n$.
}\end{thm}

Given a character $\underline{\lambda}:= (\lambda_1, \cdots, \lambda_n) : (t_1,\cdots, t_n) \to t_1^{\lambda_1} \cdots t_n^{\lambda_n}$, 
define the corresponding character on $B$ which is trivial on $U$ and consider 
$$\Ind_B^G \underline{\lambda} 
:= \left\{f:G \to \bar{\F}_p: f(bg) = \underline{\lambda}(b) f(g) \right\} ;$$
the functions $f: G \rightarrow \bar{\F}_p$ here are algebraic functions on $G$.  The group $G$ operates on such functions by right translation.

\begin{thm}{\sf 
\begin{enumerate}
\item For 
a character $\underline{\lambda}:= (\lambda_1, \cdots, \lambda_n) 
: (t_1,\cdots, t_n) \to t_1^{\lambda_1} \cdots t_n^{\lambda_n}$ of $T$,
$\Ind_B^G \underline{\lambda}$ is always a  finite dimensional vector space over $k$, and is nonzero if and only if $\omega_0(\underline{\lambda})$ 
is dominant integral where $\omega_0$ is the unique element in the Weyl group  $W= N(T)/T$ which takes all the positive roots 
of $T$ in $B$ to negative roots. (For $\GL_n$, $\omega_0(\underline{\lambda}) = (\lambda_n, \cdots, \lambda_1) $.) 
\item $\Ind_B^G \underline{\lambda}$ has a unique $U$-fixed vector.
\item $\Ind_B^G\underline{\lambda}$ has a unique irreducible submodule, denoted by $L_{\lambda}$.
\end{enumerate}
}\end{thm}

\begin{proof}{\sf 
We only prove parts (2) and (3).
It is well-known that unipotent groups always have a fixed point.
For uniqueness it suffices to prove that $\pi^{U^-}$ is one-dimensional where $U^-$ is the group of lower triangular matrices.
But $U^-B$ is an open dense set (known as the open Bruhat cell) so a function on $G$ is determined by its restriction to $U^-BU$.
This proves part (2).

\v2
Part (3) follows from part (2) since a unipotent group over a field $k$ always has a fixed vector in any algebraic representation over $k$.
}\end{proof}

\begin{lem}{\sf 
Any irreducible highest weight module of weight $\underline{\lambda}$ is the one which appears in part (3) of the above theorem.
}\end{lem}

\begin{proof}{\sf 
If $\pi$ is an irreducible representation with highest weight $\underline{\lambda}$ then we construct a map $\pi \to \Ind_B^G\omega_0(\underline{\lambda})$.
This will do the job.

Construction of $\pi \to \Ind_B^G \omega_0(\underline{\lambda})$ is nothing but Frobenius reciprocity, and depends on the
observation that the dual representation $\pi^\vee$ is a highest weight module for weight $\omega_0(\underline{\lambda})^{-1}$.
We will omit a  proof of this simple observation on $\pi^\vee$.
Now, let $v_0$ be a highest weight vector in the dual space $\pi^\vee$, and let $\langle -,- \rangle$ 
be the canonical
 bilinear form $\langle -,- \rangle: \pi \times \pi^\vee \to k$.
Then $\pi \to \Ind_B^G \omega_0(\underline{\lambda})$ is given by $v \mapsto f_v(g) = \langle g^{-1}v_0, v\rangle$ 
}\end{proof}

\begin{thm}[\sf Borel-Weil]{\sf \em
$\Ind_B^G\underline{\lambda} = H^0(\underline{\lambda})$  is already irreducible in character $0$.
}\end{thm}

\v2
\noindent{\bf Cartan Weyl theory:}
This works for $G = \GL_n(\C)$ or the compact group $\Un_n$.
For integers $\underline{\lambda}:= \lambda_1 \geq \lambda_2 \geq \cdots \geq \lambda_n$ there exists a unique irreducible representation of $\GL_n(\C)$ such that its character on the diagonal torus  is given by
$$\theta_{\underline{\lambda}}(t) = \frac{\det 
\begin{pmatrix} t_1^{\lambda_1 + n - 1} & \cdots & t_n^{\lambda_1 + n- 1} \\ 
t_1^{\lambda_2 + n - 2} & \cdots & t_n^{\lambda_2 + n - 2} \\
\vdots & & \vdots \\ 
t_1^{\lambda_n} & \cdots & t_n^{\lambda_n} \end{pmatrix}}
{\det \begin{pmatrix} t_1^{n - 1} & \cdots & t_n^{n- 1} \\ 
t_1^{n - 2} & \cdots & t_n^{n - 2} \\
\vdots & & \vdots \\ 
1 & \cdots & 1 \end{pmatrix}}$$

There is such a theorem for all reductive algebraic groups over $\C$.

\begin{thm}
The irreducible algebraic representations of $\GL_n(\bar{\F}_p)$ are also parametrized by integers $\lambda_1 \geq \cdots \geq \lambda_n$, however its character or dimension is not so easy to describe. 
\end{thm}

Lusztig has conjectured the character theory of $L_{\lambda}$, equivalently, reduction mod $\ell$ of the Weyl module $\Ind_B^G \lambda$ in characteristic $0$.

\v2
\begin{thm}[\sf Steinberg]{\sf
Let $G$ be a semi-simple algebraic group over $\F_p$. Then any irreducible representation of $G(\F_p)$ in characteristic $p$ is obtained by 
restricting an irreducible algebraic representation $\pi_\lambda$ of $G(\bar{\F}_p)$ where the highest weight $\lambda = \lambda_1\omega_1+\lambda_2\omega_2+ \cdots + \lambda_n \omega_n$, with $\omega_i$ the fundamental weights of $G$, and $0 \leq \lambda_i \leq (p-1)$. 

}\end{thm}

Finally, let us use these algebraic group theories to understand composition series of reduction $\pmod{p}$ of a principal series.

\v2
Recall that (abstract) principal series  representations  are defined as 
$$\Ind_{B(\F_q)}^{\GL_2(\F_p)}(\chi_1, \chi_2) 
= \left\{f:G(\F_p) \to \bar{\F}_q:f(bg) = \chi(b)f(g)\right\} .$$
There is a natural map 
$\Ind_{B(\bar{\F}_p)}^{\GL_2(\bar{\F}_p)} 
(\bar{\chi}_1, \bar{\chi}_2) \to \Ind_{B(\F_p)}^{\GL_2(\F_p)} (\chi_1, \chi_2)$, which consists in restricting algebraic 
functions $f$ on ${\GL_2(\bar{\F}_p)} $ with $\left \{f(bg) = \chi(b) f(g) | b \in {B(\bar{\F}_p)}, g \in  {\GL_2(\bar{\F}_p)} \right \}$ 
to abstract functions on $\GL_2(\F_p)$. Since $\omega_0(i_1,\i_2) = (i_2, i_1)$, the  space of algebraic functions, 
$\Ind_{B(\bar{\F}_p)}^{\GL_2(\bar{\F}_p)}(i_1, i_2)$ is nonzero if and only if $i_2 \geq i_1$, and the earlier remark
gives the arrow $\det^{i_1} \otimes \Sym^{i_2-i_1} V \rightarrow  \Ps(i_1, i_2) $ in the exact sequence:  
$$ 0 \rightarrow \det^{i_1} \otimes \Sym^{i_2-i_1} V \rightarrow  
\Ps(i_1, i_2) \rightarrow  \det^{i_2} \otimes \Sym^{(p-1) - (i_2-i_1)} V \rightarrow 0, ~~~{\rm if  }~~~~i_2 \geq  i_1 . $$
To get the second arrow, $\Ps(i_1, i_2) \rightarrow  \det^{i_2} \otimes \Sym^{(p-1) - (i_2-i_1)} V$, 
note that the dual of
$\Ps(i_1,i_2)$ is $\Ps(-i_1, -i_2) = \Ps(p-1-i_1,p-1-i_2)$ whereas the dual of 
$\Ind_B^G \underline{\lambda}$ is  $\Ind_B^G \omega_0(\underline{\lambda})$. The 
the second arrow, $\Ps(i_1, i_2) \rightarrow  \det^{i_2} \otimes \Sym^{(p-1) - (i_2-i_1)} V$, is obtained be dualizing the first arrow
$\det^{i_1} \otimes \Sym^{i_2-i_1} V \rightarrow  \Ps(i_1, i_2) $.

This justifies the  exact sequence for the principal series which arises in 
Lemma 4.1 by exhibiting a natural algebraic sub-representation of an abstract principal
series representation of $\GL_2(\F_p)$. A rather non-obvious 
assertion may be noted along the way that the map 
$\Ind_{B(\bar{\F}_p)}^{\GL_2(\bar{\F}_p)} 
(\bar{\chi}_1, \bar{\chi}_2) \to \Ind_{B(\F_p)}^{\GL_2(\F_p)} (\chi_1, \chi_2)$, which consists in restricting algebraic 
functions $f$ on ${\GL_2(\bar{\F}_p)} $ 
to abstract functions on $\GL_2(\F_p)$ is injective!

To get the Jordan-H\"older factors 
for the cuspidal representations which appears in Lemma 4.2, we 
refer to the geometric realization of discrete series in the cohomology of 
the projective curve $X$
$$XY^p + X^p Y - Z^{p+1} = 0,$$
which is the first example of a Deligne-Lusztig variety. 
This example is worked out in [HJ]  from the point of view of 
crystalline cohomology. It seems interesting to calculate $H^0(X, \Omega^1)$ as well $H^1(X,O_X)$ as a module for 
$\GL_2(\F_p) \times \F_{p^2}^\times$ in characteristic $p$.
}

\section{Remarks on mod $\ell \not = p$.}

{\sf One reason why the theory mod $\ell$,  $\ell \not = p$, tends to be much easier
is because the elaborate theory of $\C$-representations of $\GL_n(k)$ due to 
Bernstein-Zelevinsky extends in most aspects for $\ell \not = p$; in particular, 
the theory of derivatives makes sense $\mod \ell$, and the Bernstein-Zelevinsky filtration of a representation of $\GL_n(k)$ when restricted to the mirabolic has the same 
structure. Using these methods, one can prove that the reduction mod $\ell$ of an irreducible, integral representation (i.e., one which leaves
a lattice invariant) is of finite length.

We recall that the theory of derivatives for representations of $\GL_n(k)$ begins by fixing a non-trivial character $\psi_0: k \rightarrow \C^\times$. 
One can in fact assume that $\psi_0: k \rightarrow \bar{\Z}_\ell^\times \subset \C^\times$, such that its reduction $\mod \ell$ gives rise to a non-trivial additive character
$\bar{\psi}_0: k \rightarrow \bar{\F}^\times_\ell$, and the following lemma is elementary to prove.

\begin{Lem} Let $P_0$ be the mirabolic subgroup of $\GL_n(k)$ consisting of matrices in $\GL_n(k)$ with last row $(0,0,\cdots, 1)$. Associated to 
$\bar{\psi}_0$, define an additive character $\bar{\psi}: U \rightarrow \bar{\F}_\ell^\times$ on the group of upper triangular unipotent matrices by
$\bar{\psi}(u) = \bar{\psi}_0(u_{1,2} + u_{2,3} +\cdots, + u_{n-1,n})$. Then, ${\rm Ind}_U^{P_0} \bar{\psi}$ is an irreducible representation of $P_0$.
\end{Lem}

\begin{cor} A supercuspidal representation of $\GL_n(k)$, with integral central character, has an integral model over  ${\mathcal O}_\ell$ a finite extension of 
${\Z}_\ell$ which is the ring of integers in a finite extension $K$ of $\Q_\ell$, 
i.e., the representation $V$ which we assume is defined over $K$, contains a free ${\mathcal O}_\ell$ submodule $L$ which is invariant under $\GL_n(k)$ and 
with $L \otimes_{{\mathcal O}_\ell}K = V$. Let $\m_\ell$ be the maximal ideal in ${\mathcal O}_\ell$. 
Then the  reduction  $L/( \m_\ell L)$  which is a representation space for  $\GL_n(k)$ (called the reduction mod $\ell$ of $V$),  
is irreducible since its restriction to $P_0$ is so.
\end{cor} 

\begin{proof} {\sf By Bushnell-Kutzko, any supercuspidal representation of $\GL_n(k)$ is obtained by induction of a finite dimensional
representation of a subgroup of $\GL_n(k)$ which contains the center of $\GL_n(k)$, and is compact modulo center. Since any representation of a finite group
has an integral model, it follows that $\pi$ with integral central character has an integral model. The rest follows from the previous lemma.}
\end{proof}

}

\section{Bad primes}
Representation theory of finite groups behaves same in all 
characteristics $\ell$ as in characteristic zero as long as $\ell$ is coprime to 
the order of the group. When dealing with $p$-adic groups such as $G = \GL_n(\Q_p)$,
there are again only finitely many  `bad' primes, where the theory differs 
from theory in characteristic zero. These are the primes that divide $|\GL_n(\F_p)| = p^{n(n-1)/2}(p^n - 1)\cdots (p-1)$. We say a few words on how this comes about.

A {\em supernatural number} is a map from the set of primes in $\Z \to \N \cup \{\infty\}$, it could be of the form $p_1^{n_1} \cdots p_k^{\infty} \cdots p_r^{n_r} \cdots$.

For a profinite group $G$, we define its pro-order to be the l.c.m. of 
$|G/N|$ as $N$ runs over a system of neighborhoods of the identity. 
For example, the pro-order of a pro-$p$ group is $p^{\infty}$.
For a locally compact group such as $\GL_n(\Q_p)$, this is the l.c.m. of pro-orders of all compact open subgroups.
The pro-order of $\GL_n(\Q_p) = |\GL_n(\F_p)| p^{\infty}$.
This can be computed by observing that any compact open subgroup of $\GL_n(\Q_p)$ is contained in a conjugate of $\GL_n(\Z_p)$.

\section{Modular Langlands correspondence}

{\sf 
In this section we discuss the Local Langlands correspondence which is proved by 
Vigneras for $\GL_n(k)$ for $\ell \ne p$, and some form of it is expected for $\ell =p$.

\begin{conj}[\sf Modular LLC] Let $k$ be a $p$-adic field.
Then there exists a surjective map $(\bar{\pi} \rightarrow {\sigma}_{\bar{\pi}})$ from irreducible modular representations of $\GL_n(k)$ to 
 $n$-dimensional semisimple modular representations of the 
Weil-Deligne  group $W_k'$ making the following reduction mod $\ell$ diagram commute, where $\bar{\pi}$ in the diagram is a 
certain (distinguished) Jordan-H\"older factor of reduction mod $\ell$ of $\pi$:

$$\begin{CD}
\pi @>>> \bar{\pi}{} \\
@VVV @VVV \\
\sigma_{\pi} @>>> {\sigma}_{\bar{\pi}}.
\end{CD}$$

In some more detail, one starts with a Galois representation 
$\sigma_\pi: W_{k} \rightarrow \GL_n(\bar{\Q}_\ell)$ 
($\ell = p$ allowed), which we assume is 
integral, i.e., can be represented by  $\sigma_\pi: W_{k} \rightarrow \GL_n(\bar{\Z}_\ell)$, then the corresponding
irreducible admissible representation of $\GL_n(k)$ on a vector space over $\bar{\Q}_\ell$ is integral, i.e., leaves a $\bar{\Z}_\ell$-lattice  
invariant, and the corresponding reduction $\mod \ell$ of the representation $\pi$ has a particular Jordan-H\"older
component $\bar{\pi}$ which is an irreducible $\mod \ell$ representation of $\GL_n(k)$.
\end{conj}

\begin{que}
Where does modular LLC stand?
\end{que}

\noindent {$\ell \ne p$} : It is completely understood for mod ${\ell}$ for $\ell \ne p$ for $G = \GL_n(k)$ and is due to Vigneras. In this case, the Langlands correspondence sets up a bijective correspondence between irreducible admissible representations of $\GL_n(k)$ and $n$-dimensional 
semi-simple representations of the Weil-Deligne group. In this case, irreducible modular representations of $W_k$ of dimension $n$ 
correspond to irreducible modular supercuspidal representations of $\GL_n(k)$; both of these objects can be lifted to characteristic zero.

It may be remarked that 
both in Vigneras' work mod $\ell$, as well as in the `usual' case over $\C$, there is a weaker
form  of the Langlands correspondence, called semi-simple Langlands correspondence, which
ignores the presence of the `Deligne part of the Weil-Deligne group'. The simplifying 
aspect  of the Weil group being that it does not see the difference between 
different components of a principal series representation induced from a cuspidal data.
\v2

\noindent{$\ell = p$} : The case mod $p$ has turned out to be much harder.
This is understood only for $\GL_2(\Q_p)$! In this case, the Langlands correspondence is not a bijection since there are many more admissible representations of $\GL_n(k)$ 
over $\bar{\F}_p$ than corresponding Galois representations. In the case of
$\GL_2(k)$, the principal series representations $\chi_1 \times \chi_2$ as well as 
$\chi_2 \times \chi_1$ which as we saw earlier are distinct representations of $\chi_1 \ne \chi_2$, but both have the same parameter which is $\chi_1 \oplus \chi_2$. We refer to the ICM article [Br] of Breuil for a survey of the subject of mod $p$ representations which continues to be
a very active and still very mysterious subject. The subject of the mod $p$ representations is, so to say, the first step of a
$p$-adic representation theory of $p$-adic groups (like real representations of real groups) which is another 
big subject now, and for which we refer to  the article of Colmez [Co] as a sample of the great developments happening  in this field.

\vspace{4mm}

\noindent{\bf Remark :} One way to go about modular LLC is to force the correspondence 
$(\bar{\pi} \rightarrow {\sigma}_{\bar{\pi}})$ 
from irreducible modular representations of $\GL_n(k)$ to 
 $n$-dimensional semisimple modular representations of the 
Weil-Deligne  group $W_k'$ such that the  diagram in conjecture $(10.1)$ commutes, i.e., beginning with the modular parameter
${\sigma}_{\bar{\pi}}$, look at all the irreducible representations $\pi$ of $\GL_n(k)$ in characteristic zero whose Langlands parameter
mod $\ell$ is ${\sigma}_{\bar{\pi}}$. (By a theorem due to Fong-Swan, parameters can be lifted to characteristic zero.)   
Given the flexibility in $\pi$ here, one might hope that there is one for which
$\bar{\pi}$ is irreducible. Declare the parameter of  reduction mod $\ell$ of {\it all} such $\pi$'s (with $\bar{\pi}$ irreducible) to be $\sigma_{\bar{\pi}}$. The proposal made here partly depends on lifting a mod $\ell$ or mod $p$ 
representation of $\GL_n(k)$ to characteristic zero; this, it seems, is expected but not known in the mod $\ell \ne p$ case, whereas for $\ell = p$, it seems to be hopelessly false as there may not even be an  irreducible 
representation in characteristic zero with finite length mod $p$ containing a particular supercuspidal representation.

\section{Reduction mod $\ell$ and Brauer theory}
The previous section on the Local Langlands correspondence for $\GL_n(k)$ 
used the notions related to reduction mod $\ell$ of Galois representations (where
$\ell = p$ is allowed). For general reductive groups $G$, the $\bar{\Q}_\ell$-smooth 
representations  of $G(k)$ are supposed to be related by the Langlands 
correspondence to Galois 
representations which take values inside the L-group ${}^LG (\bar{\Q}_\ell)$, and 
one must define an appropriate notion of reduction mod $\ell$ of Galois 
representations $\sigma_\pi: W_{k} \rightarrow {}^LG(\bar{\Q}_\ell)$, 
which should now be  a `semi-simple representation' $\sigma_\pi: W_{k} \rightarrow {}^LG(\bar{\F}_\ell)$. (By a semi-simple representation $\sigma: F \rightarrow G(\bar{k})$ of an
abstract group $F$ inside an algebraic group $G$,  we mean one whose image can be 
conjugated to land inside a Levi subgroup of any parabolic in which the image of $\sigma$ lies.)

The first observation regarding $\sigma_\pi: W_{k} \rightarrow {}^LG(\bar{\Q}_\ell)$, 
with values inside a finite extension $E$ of $\Q_\ell$, is that since 
we are allowed to take further extensions of $E$, any bounded subgroup of 
${}^LG(\bar{\Q}_\ell)$ can be conjugated  to sit inside a  maximal parahoric subgroup $Q({\O}_E)$ instead of the 
more obvious, the normalizer of a parahoric subgroup. (This subtlety is relevant only when ${}^LG$ is not 
semi-simple and simplyconnected.)
 Taking the maximal reductive quotient of $Q(\O_E)$ gives us 
a subgroup of ${}^LG(\bar{\F}_\ell)$, and therefore we get $\bar{\sigma}_\pi: W_k \rightarrow {}^LG(\bar{\F}_\ell)$. 
We need to 
`semi-simplify' this homomorphism.    For this, suppose $U$,  the unipotent radical 
of a parabolic $P$ inside ${}^LG(\bar{\F}_\ell)$,  is left invariant under 
the map $\bar{\sigma}_\pi: W_k \rightarrow {}^LG(\bar{\F}_\ell)$, i.e., 
$U$ is left invariant under the conjugation action of $\bar{\sigma}(w)$ for $w \in W_k$.
Choose such a 
$U$ with the corresponding parabolic $P$ to be minimal possible for this property.
Since the normalizer of $U$ is $P$,  the map, $\bar{\sigma}_\pi: W_k \rightarrow {}^LG(\bar{\F}_\ell)$ lands inside $P(\bar{\F}_\ell)$, and then going modulo $U$,
we get a  map, $\bar{\sigma}_\pi: W_k \rightarrow M(\bar{\F}_\ell)$, 
which seems
to have the right to be called the semi-simplification of 
the map, $\bar{\sigma}_\pi: W_k \rightarrow {}^LG(\bar{\F}_\ell)$, and will be called the reduction mod $\ell$ of $\sigma_\pi: W_{k} \rightarrow {}^LG(\bar{\Q}_\ell)$. 
A generalization of Brauer theory is called for to prove that this 
$\bar{\sigma}_\pi: W_k \rightarrow M(\bar{\F}_\ell)$ is independent of the 
choices made (up to conjugation in ${}^LG(\bar{\F}_q)$), and this then is 
the  
reduction mod $\ell$ of $\sigma_\pi: W_{k} \rightarrow {}^LG(\bar{\Q}_\ell)$

}

\section{Bibliography}
[BL] Barthel, L., Livne, R.: {\it Irreducible modular representations of $\GL_2$ of a local field.} Duke Math. J. 75(2), 261-292 (1994).

[Br] Breuil, C.: {\it The emerging $p$-adic Langlands programme}, Proceedings of I.C.M. 2010, Vol. II, 203-230.

[Co] Colmez, P.: {\it Représentations de $GL_2(\Q_p)$ et $(\phi,\Gamma)$-modules.} Ast\'erisque 330, 281-509 (2010).
 
[DL] Deligne, P.; Lusztig, G. {\it Representations of reductive groups over finite fields.} Ann. of Math. (2) 103 (1976), no. 1, 103-161.

[HT] Harris, M.,\; Taylor, R. {\it On the geometry and cohomology of some simple Shimura varieties,} Annals of Mathematics studies, no. 151, Princeton University Press, Princeton, NJ.

[HJ] Haastert, B.;  Jantzen, J.C. {\it Filtrations of the Discrete Series of $\SL_2(\F_q)$ via Crystalline
Cohomology}, J. Algebra 132, 77-103 (1990).

[Hen] Henniart, G. {\it Une preuve simple des conjectures de Langlands pour $\GL_n$ sur un corps p-adique}, Invent. Math. 130 (2000) 439-455.

[Her] Herzig, F. {\it The classification of irreducible admissible mod p representations of a $p$-adic $\GL_n$}
Invent. Math. 186 (2011), no. 2, 373-434.

[L] Lusztig, G. {\it Characters of reductive groups over a finite field.} Annals of Mathematics Studies, 107. Princeton University Press, Princeton, NJ, 1984.

[V] Vign\'eras, M.-F. {\it  Représentations $\ell$-modulaires d'un groupe 
r\'eductif $p$-adique avec $\ell \not = p$.} 
Progress in Mathematics, 137. Birkhäuser Boston, Inc., Boston, MA, 1996.

[V2]  Vign\'eras, M.-F.: {\it Correspondance de Langlands semi-simple pour $\GL(n,F)$ modulo $\ell \not = p$.}
Invent. Math. 144 (2001), no. 1, 177-223. 

}
\end{document}